\documentclass[12pt]{amsart}
\usepackage{amssymb,amsmath,mathtools}

\theoremstyle{plain}

\newtheorem{theorem}{Theorem}[section]
\newtheorem{lemma}[theorem]{Lemma}
\newtheorem{proposition}[theorem]{Proposition}

\newtheorem{corollary}[theorem]{Corollary}

\theoremstyle{definition}
\newtheorem{definition}[theorem]{Definition}
\newtheorem{remark}[theorem]{Remark}

\theoremstyle{remark}

\numberwithin{equation}{section}
\numberwithin{figure}{section}
\numberwithin{table}{figure}

\usepackage{todonotes}

\newcommand*{\rom}[1]{\expandafter\@slowromancap\romannumeral #1@}

\newcommand{\pt}[1]{\left({#1}\right)}
\newcommand{\pq}[1]{\left[{#1}\right]}

\newcommand{\ps}[2]{\left\langle{#1},{#2}\right\rangle}
\newcommand{\rest}[2]{\left.{#1}\right|_{#2}}
\newcommand{\pg}[1]{\left\{{#1}\right\}}
\newcommand{\abs}[1]{\left|{#1}\right|}

\newcommand{\alb}[1]{\,\alpha^{\overline{#1}}}
\newcommand{\al}[2]{\,\alpha^{#1\bar{#2}}\,}
\newcommand{\aldue}[3]{\,\alpha^{#1\bar{#2}\bar{#3}}\,}

\newcommand{\lambdah}[1]{\Lambda^{{#1}}_{\mathfrak h}}
\newcommand{\lambdahh}[2]{\Lambda^{{#1},{#2}}_{\mathfrak h}}

\newcommand{\lambdakk}[3]{\Lambda^{{#2},{#3}}_{\mathfrak {#1}}}
\newcommand{\vol}[1]{\text{Vol}_{\mathfrak{#1}}}

\newcommand{\R}{\mathbb{R}}

\newcommand{\C}{\mathbb{C}}

\newcommand{\diffh}{d_{\mathfrak h}}

\newcommand{\del}{\partial\!}

\title{A note on $p$-K\"ahler structures on compact quotients of Lie groups}
\author{Anna Fino}
\address[Anna Fino]{Dipartimento di Matematica ``G. Peano'', Universit\`{a} degli studi di Torino \\
Via Carlo Alberto 10\\
10123 Torino, Italy\\
\& Department of Mathematics and Statistics, Florida International University\\
Miami, FL 33199, United States}
\email{annamaria.fino@unito.it, afino@fiu.edu}

\author{Asia Mainenti}
\address[Asia Mainenti]{Dipartimento di Matematica ``G. Peano'', Universit\`{a} degli studi di Torino \\
Via Carlo Alberto 10}
\email{asia.mainenti@unito.it}

\keywords{$p$-K\"ahler structure, nilmanifold, almost abelian solvmanifold}

\date{\today}

\begin{document}

\begin{abstract}
A  $p$-K\"ahler  structure  on a  complex manifold  of complex dimension $n$  is  given by  a $d$-closed  transverse  real $(p,p)$-form.      In the paper we study  the existence of $p$-K\"ahler structures  on compact quotients of simply connected Lie groups by discrete subgroups endowed with an invariant complex structure. In particular, 
we  discuss the existence of $p$-K\"ahler structures on nilmanifolds, with a focus on the case $p =2$ and  complex dimension $n = 4$.  
Moreover, we prove that a $(n-2)$-K\"ahler  almost abelian solvmanifold of complex dimension  $n\geq3$  has to be K\"ahler.
\end{abstract}
    
    \maketitle
    
\section{Introduction}

A  $p$-K\"ahler structure   on a  complex manifold $(X,  J)$ of complex dimension $n$  is  given by  a $d$-closed  transverse  real $(p,p)$-form $\Omega$.   
The $p$-K\"ahler  structures have been introduced and studied in \cite{AA, AA2, AB}.  Recently, their behavior under small deformations of the complex structure has been studied in \cite{RWZ}.

Some obstructions to their existence were determined in \cite{hmt21}, where the authors extended the definition to non-integrable almost complex manifolds, and in \cite{Sferruzza-Tardini}, on nilmanifolds with nilpotent complex structures.
In \cite{AB2}, Alessandrini and Bassanelli conjectured that if $X$ is $p$-K\"ahler then it is $q$-K\"ahler for all $p\leq q<n$.

For $p=1,n-1$, transversality is equivalent to positive definiteness, so in the first case we find the K\"ahler condition, whereas in the latter this property is equivalent to the balanced one.
In  complex dimension $3$ these are all the possible cases, and both have been thoroughly studied, so we will consider higher dimension, where more cases arise.

Examples of $2$-K\"ahler structures on {compact non-K\"ahler} complex manifolds were constructed  in  complex dimension $5$, {using}  a smooth proper modification of  $\mathbb P^5$  {\cite[Section 4]{AB2}}, and lately it was proven that non--compact examples exist  in any dimension greater than $2$, as products of $\C^m$ and a balanced complex manifold of  complex dimension $3$ \cite[Theorem 5.3]{AlessandriniLieGroups}.
On the other hand, the existence of a $2$-K\"ahler structure on a compact  (non-K\"ahler) complex manifold  of complex dimension $4$ is quite restrictive. 
As a far as we know no{  examples}  in the literature are known.

In Section \ref{preliminaries}, after recalling a few definitions and some known results on $p$-K\"ahler structures, we use the symmetrization process described in \cite{Belgun}  {(see also \cite{FG, Ugarte}),} to prove that on compact quotients of simply connected Lie groups by lattices, {endowed with  invariant complex structures,}  the existence of $p$-K\"ahler structures implies the existence of invariant ones (Lemma \ref{symmetrization}). 
{By invariant $p$-K\"ahler structure we mean one induced by a left-invariant  one on $G$ or, equivalently, by a $p$-K\"ahler structure on the Lie algebra $\mathfrak g$ of $G$.}
Furthermore, we find some necessary conditions to the existence of {$p$-K\"ahler} structures on Lie algebras  endowed {with a complex structure $J$   such that there exists  a  $J$-invariant ideal of codimension $2$} (Proposition \ref{jinvideal}), that we will use in the nilpotent case.

Section \ref{secNilmanifolds} is devoted to the nilpotent case.
We show some obstructions to the existence of $p$-K\"ahler structures when the {nilmanifold} is endowed with a {quasi-nilpotent} complex structure $J$, namely such that the center {of the associated  nilpotent Lie algebra} has a non-trivial $J$-invariant subspace, with a focus on $p=2$ (Proposition \ref{quasinilppK}). 
Later on, we study  nilmanifolds {of complex dimension $4$} and prove that they do not admit $2$-K\"ahler structures,  unless they are tori (Theorem \ref{no2k8}).
We use this result as the first step of induction to prove that a nilmanifold  of  complex dimension {greater then 3} endowed with an invariant  quasi--nilpotent  complex structure cannot admit  $2$-K\"ahler structures unless it is a torus {(Theorem \ref{quasinilpNON2k}).}

Finally, we  consider the almost abelian case. 
We recall that a Lie group is called almost abelian if its Lie algebra has a codimension one abelian ideal.
In Section \ref{secAlmostAb} we  prove that   for almost abelian solvmanifolds of complex dimension  $n\geq3$, the $(n-2)$-K\"ahler condition implies K\"ahler  (Theorem \ref{ThmAlmostAb}). 
This gives yet another case in complex dimension $4$   of compact complex manifolds  for which the existence of a  $2$-K\"ahler structure forces the manifold to be K\"ahler.

\section{Preliminaries on $p$-K\"ahler structures}\label{preliminaries}

Let $V$ be a complex vector space of dimension $n$ and let us denote by $\Lambda^{p,q}: =\Lambda^{p,q}V^*$ the space of $(p, q)$ forms over $V$.
In the following lines we will recall a few positivity notions for differential forms.

\begin{definition}\label{poscond}
\begin{enumerate}
    \item A  $(n, n)$-form $\nu$  on $V$ is \textit{positive} if $\nu= c\,i\alpha^1\wedge\alpha^{\bar1}\wedge\cdots\wedge i\alpha^n\wedge\alpha^{\bar{n}}$, where $c\in\R_{\geq0}$ and $\pg{\alpha^j}_{j=1}^n$ is a basis of $\Lambda^{1,0}$.
    If $c>0$, we will call $\nu$ a \textit{volume form}.
    \item A $(q,0)$-form $\psi$  on $V$  is called \textit{simple} if $\psi=\mu_1\wedge\dots\wedge\mu_q$, with $\mu_1,\dots,\mu_q\in\Lambda^{1,0}$.
    A $(k,k)$-form $\Omega$ is called \textit{transverse} if for every {non--zero} simple $(n-k,0)$-form $\psi$,
    $$
    i^{\pt{n-k}^2}\,\Omega\wedge\psi\wedge\bar\psi
    =\Omega\wedge i\mu_1\wedge\bar\mu_1\wedge\dots\wedge i\mu_{n-k}\wedge\bar\mu_{n-k}
    $$
    is a volume form.
    \item A $(k,k)$-form $\Omega$ on $V$   is said to be \textit{positive definite} if for all $ {0\neq}\eta\in\Lambda^{n-k,0}$, the $(n,n)$-form 
    $$
    i^{\pt{n-k}^2}\,\Omega\wedge\eta\wedge\bar\eta
    $$
    is a volume form. 
    \item A $(k,k)$-form $\Omega$  on $V$  is \textit{strongly positive} if it can be written as 
    $$\Omega=i^{{k}^2}\sum_j\psi_j\wedge\bar{\psi}_j,$$
    with $\psi_j\in\Lambda^{k,0}$ simple.
\end{enumerate}
\end{definition}

It follows from Definition  \ref{poscond},  that every strongly positive form is positive definite and every positive definite form is transverse.
Moreover every transverse form is real.
Notice that for $k=1, n-1$, a $(k,0)$-form is always simple (see for example \cite{hk}).
In fact, for $k=1, n-1$,  (2), (3) and (4) in Definition \ref{poscond} are all  equivalent. 

\smallskip

Using the previous notion of tranversality one can introduce the following 
\begin{definition}

Let $(X, J)$  be complex manifold  of complex dimension $n$ and let $1 \leq p < n$.  A \textit{$p$-K\"ahler}  {structure}  on $X$ is   given by  a $d$-closed  real $(p,p)$-form  $\Omega$ such that,  
at every point $x\in X$, $\Omega_x \in  \Lambda^{p,p} (T_x X)$ is  transverse.
\end{definition}

For $p=1$ a transverse form is nothing but  the fundamental form associated to a Hermitian metric on $(X,J)$.
This means that a $1$-K\"ahler structure actually gives a K\"ahler metric on $(X,J)$.
On the other hand, when $p=n-1$, we know by \cite{michelsohn} that every strongly positive $(n-1,n-1)$-form can be written as the $(n-1)$-th power of a strictly positive $(1,1)$-form and so determines a Hermitian metric.
It follows that the datum of a $(n-1)$-K\"ahler structure is equivalent to that of a balanced Hermitian metric.
Note that these are the only two cases where $p$-K\"ahler structures have metric meaning (cf. \cite[Proposition 2.1]{AB}).

The following result gives an obstruction to the existence of {$p$-K\"ahler} structures on compact complex manifolds. 
\begin{proposition}[\cite{hmt21}] \label{obstructpK}
   Let $(X,J)$ be a compact complex manifold.  Suppose there exists a $(2n -2p -1)$-form  $\beta$  on $X$ such that   \begin{equation*}
        0\neq\pt{d\beta}^{n-p,n-p}=\sum_jc_j\,\psi_j\wedge\overline{\psi_j},
    \end{equation*}
    where $c_j\in\R$ have the same sign and $\psi_j$ are simple $(n-p,0)$-forms.
    Then $(X,J)$ does not admit any {$p$-K\"ahler} structure.
\end{proposition}

\smallskip

Let us now consider as a complex manifold  $(X, J)$   the  compact quotient  $X = \Gamma \backslash G$   of a simply connected  Lie group $G$  by a  discrete subgroup $\Gamma$ endowed with  an invariant complex structure  $J$, i.e. a complex structure  induced by a complex structure on $\mathfrak g$.

{
{Next we prove that the existence of a $p$-K\"ahler structure   on  $(X = \Gamma \backslash G, J)$ implies the existence of an invariant one. 

\begin{lemma}\label{symmetrization}
If $(X = \Gamma \backslash G, J)$  admits a $p$-K\"ahler structure $\Omega$,  then $(\mathfrak g, J)$  has  a $p$-K\"ahler structure.
\end{lemma}

\begin{proof}
Let $\nu$ be a volume element on $X$ induced by a bi-invariant one on the Lie group G (the existence of such a volume form was proved in \cite{Milnor}).
After rescaling, we can suppose that $X$ has volume equal to 1.
Given the $p$-K\"ahler structure $\Omega$,
 by symmetrization,  we can define  the $(p,p)$-form $\Omega_{\nu}$ on $\mathfrak g$,  given by
 $$
 \Omega_{\nu} (Y_1, \ldots, Y_{2p}) = \int_{x \in X} \Omega_x (Y_1 \vert_x, \ldots,  Y_{2p} \vert_x) \nu,
 $$
for every $Y_j \in \mathfrak g$, where $Y_j \vert_x$ is the value at the point  $x \in X$  of the projection on $X$ of the left-invariant vector field $Y_j$ on the Lie group  $G$.
By \cite{Belgun} (see also \cite{FG, Ugarte})  the symmetrization commutes with the differential $d$, so  $d \Omega_{\nu} =0$.
We only need to show that $\Omega_{\nu}$ is still transverse, i.e. that  $
    i^{\pt{n-p}^2}\,\Omega_{\nu} \wedge\psi\wedge\bar\psi
  $ is a volume form on $\mathfrak g$, 
for every simple $(n-p,0)$-form $\psi$ on $\mathfrak g$.
This follows from the fact that for every differential forms $\alpha$ and $\beta$  on $X$ we have
$$
(\alpha_{\nu} \wedge \beta)_{\nu} = \alpha_{\nu}  \wedge \beta_{\nu}.
$$
In fact, using $\psi=\psi_{\nu}$ we find that $i^{\pt{n-p}^2}\,\Omega_{\nu} \wedge\psi\wedge\bar\psi$ is a volume form, as it is the integral on $X$ of $i^{\pt{n-p}^2}\,\Omega\wedge\psi\wedge\bar\psi$, positive by hypothesis.
\end{proof}}}

We now state some general restriction to the existence of a $p$-K\"ahler structure on a Lie algebra.
{From now on $\mathfrak g^{1,0}$ (respectively $\mathfrak g^{0,1}$) will denote the $i$-eigenspace (respectively $-i$-eigenspace) of $J$ as an endomorphism of $\mathfrak g^*$.}

\begin{proposition}\label{jinvideal}
    If a $p$-K\"ahler Lie algebra $(\mathfrak g,J, \Omega)$  {of complex dimension $n \geq 3$}, with $p<n-1$, admits a closed non-zero $(1,0)$-form $\alpha$, then it has a $p$-K\"ahler $J$-invariant ideal of codimension 2.
\end{proposition}  

\begin{proof}
    Let   $\{ \alpha^1, \ldots \alpha^n \} $ be  a basis  of $\mathfrak g^{1,0} $ such that $\alpha^1:= \alpha$.
    Consider its dual basis  $\{Z_1, \ldots, Z_n \} $ .
    Consider a basis $\pg{e_1,\dots,e_{2n}}$ of $\mathfrak g$ such that $Z_1=e_1-i\,e_2$.
    Clearly, $Je_1=e_2$ and the dual elements of $e_1,e_2$ in $\mathfrak g^*$ are closed differential 1-forms on $\mathfrak g$, so that the subspace $\mathfrak h:= \mbox{span}\pg{ e_3\dots,e_{2n}}$ is a $J$-invariant ideal of $\mathfrak g$ and $\mathfrak h^{1,0}$ is generated by ${\alpha^2,\dots,\alpha^{n}}$. 
    Let us denote with $\diffh$ the exterior derivative of $\mathfrak h$.
    Then, for every form $\beta\in\Lambda_{\mathfrak h}$ one has that $d\beta=\diffh\beta+\tilde d\beta$, with $\tilde d\beta\in I\pt{\alpha^1,\alpha^{\bar 1} }$.
    Let $\Omega$ be a $p$-K\"ahler structure  on $(\mathfrak g, J)$.
    We will prove that its restriction $\Omega_{\mathfrak h}\in\lambdahh pp$ to $\mathfrak h$ is a $p$-K\"ahler structure on $(\mathfrak h, J \vert_{\mathfrak h}$).  
    We can write 
    \begin{equation*}
    \Omega=\Omega_{\mathfrak h}+\alpha^1\wedge\eta +\alpha^{\bar1}\wedge\bar\eta +i\al11\wedge\vartheta,
    \end{equation*}
    with $\eta\in\lambdahh {p-1}p$, $\vartheta\in\lambdahh {p-1}{p-1}$, $\Omega_{\mathfrak h}$ and $\vartheta$ real. Therefore
    \begin{equation*}
    \begin{split}
        d\Omega&=d\Omega_{\mathfrak h}-\alpha^1\wedge d\eta -\alpha^{\bar1}\wedge d\bar\eta +i\al11\wedge d\vartheta\\
            &=\diffh\Omega_{\mathfrak h}+\tilde d\Omega_{\mathfrak h}-\alpha^1\wedge d\eta -\alpha^{\bar1}\wedge d\bar\eta +i\al11\wedge d\vartheta,
    \end{split}
    \end{equation*}
    with $\diffh\Omega_{\mathfrak h}\in\lambdah{2p+1}$ and $d\Omega-\diffh\Omega_{\mathfrak h}\in I\pt{\alpha^1,\alpha^{\bar 1} }$, so that if $\Omega$ is $d$-closed, $\Omega_{\mathfrak h}$ must be  $\diffh$-closed.
    It remains to prove that $\Omega_{\mathfrak h}$ is transverse, namely that,  for all $$\phi=i\mu_1\wedge\overline{\mu_1}\wedge\dots\wedge i\mu_{n-1-p}\wedge\overline{\mu_{n-1-p}},  \quad \mu_j\in\mathfrak h^{1,0},$$  the $(n-1,n-1)$-form $\Omega_{\mathfrak h}\wedge \phi$ is a volume form on $\mathfrak h$.
    Since $\Omega$ is transverse, $\Omega\wedge i\al11\wedge\phi$ is a volume form on $\mathfrak g$.
    One easily sees that
    $$
    \Omega\wedge i\al11\wedge\phi=\Omega_{\mathfrak h}\wedge i\al11\wedge\phi,
    $$
    yielding that $\Omega_{\mathfrak h}\wedge \phi$ is a volume form on $\mathfrak h$, as wanted.
\end{proof}

\section{$p$-K\"ahler  structures on nilmanifolds}\label{secNilmanifolds}

We will now discuss the case where $X$ is a {nilmanifold}, i.e. a compact  quotient $\Gamma\backslash G$ of a simply connected nilpotent  Lie group $G$ by a lattice $\Gamma$ {endowed with am invariant complex structure.}
Lemma \ref{symmetrization} allows us to restrict to the study of $p$-K\"ahler structures on the nilpotent Lie algebra  $\mathfrak g$ of $G$.
From \cite{Salamon}, we know that for every complex structure $J$ on a nilpotent Lie algebra   $\mathfrak g$, there exists a basis $\{\alpha^1, \ldots, \alpha^n\} $ of $\mathfrak g^{1,0}$ such that for all $j$,
\begin{equation*}
    d\alpha^{j+1}\in I\pt{\alpha^1,\dots,\alpha^j}.
\end{equation*}
In particular, $d\alpha^1=0$, so Proposition \ref{jinvideal} holds.

We recall that given a complex structure $J$ on a real nilpotent Lie algebra $\mathfrak g$ of dimension $2n$, one can define the {ascending series adapted to $J$} as follows
\begin{align*}
{\mathfrak a}_0(J)&=\pg{0},	\\
{\mathfrak a}_{k}(J)&=\pg{X\in\mathfrak g\colon \pq{X,\mathfrak g}\subseteq {\mathfrak a}_{k-1}(J),\pq{JX,\mathfrak g}\subseteq {\mathfrak a}_{k-1}(J)}, 	\quad\text{for }k\geq1.
\end{align*}
Then $J$ is said to be:
\begin{itemize}
    \item \textit{strongly non-nilpotent (SnN)} if ${\mathfrak a}_{1}(J)=\pg{0}$;
    \item \textit{quasi-nilpotent} if ${\mathfrak a}_{1}(J)\neq\pg{0}$. In this case, the ascending series adapted to $J$ stabilizes, namely there exists a positive integer $t$ such that ${\mathfrak a}_{t}(J)={\mathfrak a}_{l}(J)$ for all $l\geq t$ and we can distinguish between two subcases:
    \begin{itemize}
        \item $J$ is \textit{weakly non-nilpotent} if ${\mathfrak a}_{t}(J)\neq\mathfrak{g}$;
        \item $J$ is \textit{nilpotent} if ${\mathfrak a}_{t}(J)=\mathfrak{g}$, or equivalently, if  there is a basis $ \{ \alpha^1, \ldots, \alpha^n \}$ of $\mathfrak{g}^{1,0}$ satisfying 
\begin{equation*}
\begin{cases} 
        d\alpha^1   =0,   \\
        d\alpha^{j} \in\Lambda^2\ps{\alpha^1,\dots,\alpha^{j-1}}{\alb{1},\dots,\alb{j-1}},& j=2,\dots,n.
\end{cases}
\end{equation*}
    \end{itemize}
\end{itemize}

The following result is a consequence of Proposition \ref{obstructpK} and it gives an obstruction to the existence of {$p$-K\"ahler} structures.

\begin{proposition}[\cite{Sferruzza-Tardini}]\label{stnilp}
    Let $\mathfrak{g}$ be a nilpotent Lie algebra of complex dimension $n$ endowed with a nilpotent complex structure $J$.
    Given a basis $\{\alpha^1, \ldots, \alpha^n \}$ of $\mathfrak{g}^{1,0}$, let $t$ be the positive integer such that
    \begin{equation*}
        d\alpha^j=0,\quad\text{for }j=1,\dots,t,  \quad  {\text{and}} \quad 
        d\alpha^j\neq0\quad\text{for }j=t+1,\dots,n.
    \end{equation*}
    Then, there are no $(n-t)$-K\"ahler structures  on $(\mathfrak{g},J)$.
\end{proposition}

\subsection{Quasi-nilpotent complex structures}\label{secQNilp}

We will now consider quasi-nilpotent complex structures, namely the case where the center $\zeta$ of $\mathfrak g$ has a $J$-invariant non-trivial subspace.
{We recall the following}

\begin{definition}[\cite{LPhD}]
Let  $\frak g$ be   {a nilpotent} Lie algebra endowed with a {quasi-nilpotent} complex structure $J$ and 
$\mathfrak b$ be a $J$-invariant subspace of real dimension $2$. 
If $\mathfrak k$ is a nilpotent Lie algebra of {real dimension $2(n-1)$} endowed  with a complex structure $K$ such that $(\mathfrak k,K)$ is isomorphic to $(\mathfrak g/\mathfrak b,\rest J{\mathfrak g/\mathfrak b})$,  the pair $(\mathfrak g,J)$ is called a $\mathfrak b$-extension of $(\mathfrak k,K)$.
\end{definition}

{We can prove the following}.

\begin{proposition}\label{quasinilppK}
    If $\mathfrak g$ is a nilpotent Lie algebra of real  dimension $2n\geq6$ endowed with a quasi-nilpotent complex structure $J$ and admitting a  $p$-K\"ahler structure, then {$(\mathfrak g, J)$} is the $\mathfrak b$-extension of a $(p-1)$-K\"ahler nilpotent Lie algebra.
\end{proposition}

\begin{proof}
    In \cite{LPhD}, it was proven that for every pair $({\mathfrak g} ,J)$, where $\mathfrak g$ is a nilpotent Lie algebra and $J$ is quasi-nilpotent complex structure, there exists a $2$-dimensional $J$-invariant subspace $\mathfrak b$ of ${\mathfrak a}_1(J)=\zeta\cap J\zeta$, where $\zeta$ is the center of $\mathfrak g$, such that $(\mathfrak g,J)$ is the $\mathfrak b$-extension of some nilpotent Lie algebra  $\mathfrak k$ of  codimension 2   endowed with a complex structure $K$.
    We can choose a basis $\{\alpha^1, \dots, \alpha^n \}$ of $\mathfrak g^{1,0} $ such that $\mathfrak b^{1,0}$ is generated by $\alpha^n$.
    In this way, $\pg{\alpha^1,\dots, \alpha^{n-1}}$ is a basis of $\mathfrak k^{1,0} $  and     $d\alpha^j\in\Lambda^2_{\mathfrak k }$, for all $j$.
    Let $\Omega$ be a $p$-K\"ahler structure  on $\mathfrak g$.
   Then there exist $\Omega_{\mathfrak k}\in\lambdakk kpp$ and $\omega\in\lambdakk k{p-1}{p-1}$ real, $\eta\in\lambdakk k{p-1}p$ such that
   \begin{equation}\label{OmegaRestrict}
    \Omega=\Omega_{\mathfrak k}+\eta \wedge\alpha^n+\bar\eta \wedge\alpha^{\bar n}+\omega \wedge i\al nn.
   \end{equation}
    We will prove that $\omega $ is a $(p-1)$-K\"ahler form for $(\mathfrak k,K)$.
    {The  closure of the $(p-1, p-1)$-form $\omega$} follows from the fact that  $\Omega$ is closed and
    \begin{equation}\label{diffpform}
    \begin{split}
        0=d\Omega=&d\Omega_{\mathfrak k}-\eta\wedge d\alpha^n-\bar\eta \wedge d\alpha^{\bar n} \\
            &+\pt{d\eta +i\omega \wedge d\alpha^{\bar n}}\wedge\alpha^n +\pt{d\bar\eta +i\omega \wedge d\alpha^n}\wedge\alpha^{\bar n}   \\
            &+d\omega \wedge i\al nn.
    \end{split}    
    \end{equation}
    To prove that $\omega $ is transverse, fix  {a simple form} $\psi\in\Lambda^{(n-1)-(p-1),0}=\Lambda^{n-p,0}$ .
    For dimensional reasons, $\Omega_{\mathfrak k}\wedge\psi\wedge\bar\psi=\eta\wedge\psi\wedge\bar\psi=0$, giving
    $$
    i^{\pt{n-p}^2}\Omega\wedge\psi\wedge\bar\psi=i^{\pt{n-p}^2}{\omega }\wedge\psi\wedge\bar\psi\wedge i\al nn.
    $$
    This is a volume form on $\mathfrak g$ because $\Omega$ is transverse, so $i^{\pt{n-p}^2}{\omega }\wedge\psi\wedge\bar\psi$ is a volume form on $\mathfrak k$.
\end{proof}

{As a consequence we have the following}

\begin{corollary}\label{QNilp2k}
{Let $\mathfrak g$ be  a  nilpotent Lie algebra  of real dimension $2n$ endowed with a quasi-nilpotent complex structure  $J$. If $(\frak g,  J)$  admits a $2$-K\"ahler structure}, then  $(\mathfrak g, J)$  must be the $\mathfrak b$-extension of the $2(n-1)$-dimensional abelian Lie algebra.
    Moreover, the complex structure $J$ has to be nilpotent.
\end{corollary}

\begin{proof}
    The first part of the statement is given by Proposition \ref{quasinilppK} for $p=2$.
    It follows that the  {complex structure equations  of $(\mathfrak g, J)$ must be} 
    \begin{equation*}
    \begin{cases}
        d\alpha^j=0,    &   j=1,\dots,n-1,   \\
        d\alpha^n\in\Lambda^2_{\mathfrak k},
    \end{cases}
    \end{equation*}
    with $\Lambda^2_{\mathfrak k}=\Lambda^2\ps{\alpha^1,\dots,\alpha^{n-1}}{\alb{1},\dots,\alb{n-1}}$, meaning that $J$ {has to be} nilpotent.
\end{proof}

\subsection{$2$-K\"ahler structures on  nilmanifolds of real  dimension $8$}

In  real dimension  $8$ we can actually prove that $2$-K\"ahler nilmanifolds endowed with an invariant complex structures {must be}   K\"ahler.

\begin{proposition}
    A  (non-abelian) $8$-dimensional nilpotent Lie algebra $\mathfrak g$  endowed with a quasi-nilpotent complex structure  $J$  does not admit $2$-K\"ahler structures.
\end{proposition}

\begin{proof}
    We only have to prove the statement for $\mathfrak b$-extensions $(\mathfrak g,J)$ of the $6$-dimensional abelian Lie algebra $\mathfrak k$.
    Namely, we can suppose to have a basis $\{ \alpha^1, \ldots , \alpha^4 \}$ of $\mathfrak g^{1,0} $ such that $d\alpha^j=0$ for $j=1,2,3$ and $d\alpha^4\in\Lambda^2_{\mathfrak k} $.
    If $\del \, \alpha^4 \neq 0$, it is in particular a $(2,0)$-form on a space of complex dimension $3$, hence simple, so that $\alpha^4\wedge\overline{\del \, \alpha^4}$ is a $3$-form as in Proposition \ref{obstructpK} and $(\mathfrak g,J)$ is not $2$-K\"ahler.
    It remains to consider the case where $d\alpha^4\in\Lambda^{1,1}_{\mathfrak k}$.
    Suppose there exists a $2$-K\"ahler form $\Omega$ on $(\mathfrak g,J)$. 
    
    We can write $\Omega$ as in \eqref{OmegaRestrict} and because $\mathfrak k$ is abelian, \eqref{diffpform} reduces to
    \begin{equation*}
        0=d\Omega=-\eta\wedge d\alpha^4-\bar\eta\wedge d\alpha^{\bar 4}{+i\,\omega\wedge d\alpha^{\bar 4}}\wedge\alpha^4 +{i\,\omega\wedge d\alpha^4}\wedge\alpha^{\bar 4}.
    \end{equation*}
    Since $-\eta\wedge d\alpha^4-\bar\eta\wedge d\alpha^{\bar 4}\in\Lambda^5_{\mathfrak k} $, we have $\omega\wedge d\alpha^4=0$.
    We already proved that since $\Omega$ is transverse, $\omega$ is a transverse $(1,1)$-form, hence strongly positive, giving a contradiction.
\end{proof}

Strongly non-nilpotent complex structures on $8$-dimensional nilpotent Lie algebras were classified in   \cite{LPhD} and then refined in  \cite{luv1}.
In particular, {it turns out that}, depending on the {ascending type} of $\mathfrak{g}$, the admissible complex structures are divided into two families as follows.

\begin{proposition}\cite[Thm. 3.3]{luv1}\label{8snn}
Let $J$ be a strongly non-nilpotent complex structure on an $8$-dimensional nilpotent Lie algebra $\mathfrak g$.
Then, there exists a basis $\{ \alpha^1, \ldots, \alpha^4 \}$ of $\mathfrak{g}^{1,0}$ such that the complex structure equations are {either given by}
\item 
    \begin{equation}\label{cs8f1}
    \begin{cases}
    d\alpha^1={0},	\\
    d\alpha^2=\varepsilon\al11,	\\
    d\alpha^3=\alpha^{14}+\al14+a\al21+i\,\delta\,\varepsilon\,b\al12,	\\
    d\alpha^4=i\,\nu\al11+b\al22+i\,\delta\pt{\al13-\al31},	
    \end{cases}
    \end{equation}
    where $\delta=\pm1$, $(a,b)\in\R^2\setminus\pg{(0,0)}$, $a\geq0$ and the tuple $\pt{\varepsilon,\nu,a,b}$ is one of the following:
    $$
    (0,0,0,1), (0,0,1,0), (0,0,1,1), \pt{0,1,0,\frac{b}{\abs{b}}}, (0,1,1,b), (1,0,0,1), (1,0,1,\abs{b}), (1,1,a,b),
    $$
    
  \noindent   or given by 
    \begin{equation}\label{cs8f2}
    \begin{cases}
    d\alpha^1={0},	\\
    d\alpha^2=\alpha^{14}+\al14,	\\
    d\alpha^3=a\al11+\varepsilon\pt{\alpha^{12}+\al12-\al21}+i\,\mu\pt{\alpha^{24}+\al24},	\\
    d\alpha^4=i\,\nu\al11-\mu\al22+i\,b\pt{\al12-\al21}+i\pt{\al13-\al31},	
    \end{cases}
    \end{equation}
    where $a,b\in\R$, and the tuple $\pt{\varepsilon,\mu,\nu,a,b}$ is one of the following:
    $$
    (1,1,0,a,b), (1,0,1,a,b), (1,0,0,0,b), (1,0,0,1,b), (0,1,0,0,0), (0,1,0,1,0).
    $$
\end{proposition}

For both of the families {\eqref{cs8f1} and \eqref{cs8f2}}, we can find a $3$-form $\beta$ as in Proposition \ref{obstructpK} that gives us an obstruction to the existence of $2$-K\"ahler forms on $(\mathfrak g,J)$.
In the case of {the family \eqref{cs8f1},}  we can take $\beta=b\al{14}1-a\al{13}2$.
This gives
$$
d\beta=(a^2+b^2)\aldue{12}12,
$$
never zero because $(a,b)\in\R^2\setminus\pg{(0,0)}$.
If $J$ belongs {to the family \eqref{cs8f2}}  we choose $\beta=\al{14}1+\pt{1-\mu}\al{12}3$, giving
$$
d\beta=(\varepsilon-\varepsilon\mu-\mu)\aldue{12}12,
$$
where $\varepsilon,\mu\in\pg{0,1}$ and $\pt{\varepsilon,\mu}\neq\pt{0,0}$, hence $d\beta$ is a non-zero $(2,2)$-form as wanted.

\smallskip

We have proved the following.

\begin{theorem}\label{no2k8}
    A nilpotent Lie algebra $\mathfrak g$ of real dimension $8$ endowed with a complex structure {admits a  $2$-K\"ahler structure} if and only if it is abelian.
\end{theorem}

This result can actually be generalized to any dimension, when $\mathfrak g$ is endowed with a quasi-nilpotent complex structure $J$.

\begin{theorem}\label{quasinilpNON2k}
    A non--abelian nilpotent Lie algebra  of  real dimension $2n\geq8$ endowed with a quasi-nilpotent complex structure cannot be $2$-K\"ahler.
\end{theorem}

\begin{proof}
    We will prove the statement by induction on $n\geq4$.
    The base of the induction is true by Theorem \ref{no2k8}.
    Suppose the theorem is proved for every nilpotent Lie algebra  of dimension $2(n-1)$ and consider $(\mathfrak g,J)$, with $\mathfrak g$ nilpotent and $J$ quasi-nilpotent, admitting a $2$-K\"ahler structure.
    We want to prove that $\mathfrak g$ is abelian.
    {From Corollary \ref{QNilp2k} we know that the {complex} structure equations must be 
    \begin{equation}\label{StrEqQNilp2kn}
    \begin{cases}
        d\alpha^j=0,    &   j=1,\dots,n-1,   \\
        d\alpha^n\in\Lambda^2_{\mathfrak k }=\Lambda^2\ps{\alpha^1,\dots,\alpha^{n-1}}{\alb{1},\dots,\alb{n-1}}.
    \end{cases}
    \end{equation}
    for some basis $\pg{\alpha^1\dots\alpha^n} $ of $\mathfrak g^{1,0}$.}
    Fix the dual basis $\pg{Z_j,\overline{Z_j}}_{j=1}^n$ of $\mathfrak g_{\C}$ dual to $\pg{\alpha^j\alpha^{\bar j}}_{j=1}^n$ and consider the ideal $\mathfrak h$ of $\mathfrak g$ such that $\mathfrak h_{\C}=\text{span}\pg{Z_2,\dots, Z_n,\overline{Z_2},\dots, \overline{Z_n}}$.
    Proposition \ref{jinvideal} also holds so $\mathfrak h$, endowed with the complex structure $\tilde J:= J \vert_{\mathfrak h}$, must be $2$-K\"ahler as well.
    If $\tilde J$ is quasi-nilpotent, we can use the inductive hypothesis on $\mathfrak h$ to conclude that it is abelian.
    From \eqref{StrEqQNilp2kn} we know that the complex structure equations of $(\mathfrak h,\tilde J)$ are
    \begin{equation*}
        \diffh\alpha^j=0,\quad j=2,\dots,n-1,\quad\quad \diffh\alpha^n\in\Lambda^2_{\mathfrak h\cap\mathfrak k}=\Lambda^2\ps{\alpha^2,\dots,\alpha^{n-1}}{\alb{2},\dots,\alb{n-1}},
    \end{equation*}
    so $\tilde J$ is actually nilpotent and $\mathfrak h$ must be abelian.
    This, together with (\ref{StrEqQNilp2kn}), gives
    $$
    d\alpha^n=\alpha^1\wedge\gamma_1+\alpha^{\bar1}\wedge\gamma_2+c\al11,
    $$
    for some $\gamma_1\in\Lambda^1_{\mathfrak h\cap\mathfrak k}$, $\gamma_2\in\Lambda^{1,0}_{\mathfrak h\cap\mathfrak k}$ and $c\in\C$.
    We already saw that a $2$-K\"ahler structure $\Omega$ on $(\mathfrak g,J)$ can be written as 
    $$
    \Omega=\Omega_{\mathfrak h}+\alpha^1\wedge\eta+\alpha^{\bar1}\wedge\overline{\eta}+i\al11\wedge\vartheta,
    $$
    with $\Omega_{\mathfrak h}\in\lambdahh22$ real and transverse, $\eta\in\lambdahh12$, $\vartheta\in\lambdahh11$.
    Moreover, $d\Omega=0$ if and only if
    \begin{equation*}
        d\Omega_{\mathfrak h}=\alpha^1\wedge d\eta+\alb1\wedge d\overline{\eta}\in I\pt{\al11} .
    \end{equation*}
    {Similarly to Proposition \ref{quasinilppK}, if $\Omega_{\mathfrak h\cap\mathfrak k}$ is the restriction of $\Omega_{\mathfrak h}$ to $\mathfrak k$,}
    we have
    $$
    \Omega_{\mathfrak h}=\Omega_{\mathfrak h\cap\mathfrak k}+\beta\wedge\alpha^n+\overline{\beta}\wedge\alpha^{\bar n}+\omega \wedge i\al nn,
    $$
    with  $\Omega_{\mathfrak h\cap\mathfrak k}\in\lambdakk {\mathfrak h\cap\mathfrak k}22$ and $\omega\in\lambdakk {\mathfrak h\cap\mathfrak k}11$ real and transverse and $\beta\in\lambdakk {\mathfrak h\cap\mathfrak k}12$.
    We get
    \begin{equation}\label{dOmega22}
    \begin{split}
        d\Omega_{\mathfrak h}&=-\beta\wedge d\alpha^n-\overline{\beta}\wedge d\alpha^{\bar n}+i\omega \wedge d\alpha^n\wedge\alb n-i\omega \wedge d\!\alb n\wedge\alpha^n  \\
            &=\alpha^1\wedge\pt{\beta\wedge\gamma_1+\overline{\beta}\wedge\overline{\gamma_2}+i\omega \wedge\gamma_1\wedge\alb n+i\omega \wedge\overline{\gamma_2}\wedge\alpha^n}   \\
            &\quad+\alb1\wedge\pt{\beta\wedge\gamma_2+\overline{\beta}\wedge\overline{\gamma_1}+i\omega \wedge\gamma_2\wedge\alb n+i\omega \wedge\overline{\gamma_1}\wedge\alpha^n}   \\
            &\quad+c\al11\wedge\pt{-\beta+\overline{\beta}\wedge\overline{\gamma_1}+i\omega \wedge\alb n-i\omega \wedge\alpha^n}.   \\
    \end{split}
    \end{equation}
    We know that $d\Omega_{\mathfrak h}$ must be in the ideal generated by $\al11$, so the second and third lines in {\eqref{dOmega22}} should vanish, giving
    \begin{equation*}
    {\beta\wedge\gamma_1+\overline{\beta}\wedge\overline{\gamma_2}+i\omega \wedge\gamma_1\wedge\alb n+i\omega \wedge\overline{\gamma_2}\wedge\alpha^n}=0.
    \end{equation*}
    Notice that $\beta\wedge\gamma_1+\overline{\beta}\wedge\overline{\gamma_2}$ lies in $\Lambda^{4}_{\mathfrak h\cap\mathfrak k}$, so this is equivalent to 
    \begin{equation}\label{eqinomegai}
    \begin{cases}
        \beta\wedge\gamma_1+\overline{\beta}\wedge\overline{\gamma_2}=0, \\
        \omega \wedge\gamma_1=0, \\
        \omega \wedge\overline{\gamma_2}=0.
    \end{cases}
    \end{equation}
    Recall that $ \omega $ is a transverse $(1,1)$-form, hence positive definite.
    A direct consequence is that the last condition in \eqref{eqinomegai} is equivalent to $\gamma_2=0$.
    This holds because if the $(1,0)$-form $\gamma_2$ is non--zero, for any fixed $\psi\in\Lambda^{n-2,0}$ the positive definiteness of $\omega$ implies
    $$
    0=\omega\wedge\gamma_2\wedge\psi\wedge\overline{\gamma_2}\wedge\overline\psi=\tilde c\,\vol{g},
    $$
    for some positive constant $\tilde c\in\R$, giving a contradiction.
    We can also prove that the second condition in \eqref{eqinomegai} is equivalent to $\gamma_1$ being zero.
    Indeed, the $1$-form $\gamma_1$ can be written as $\gamma_1^{1,0}+\gamma_1^{0,1}$, and $\omega\wedge\gamma_1=0$ is equivalent to
    $$
    \omega\wedge\gamma_1^{1,0}=0,\quad\omega\wedge\gamma_1^{0,1}=0.
    $$
    The same argument used for $\gamma_2$ can then be used for $\gamma_1^{1,0}$ and $\gamma_1^{0,1}$, to conclude that \eqref{eqinomegai} implies $\gamma_1^{1,0}=\gamma_1^{0,1}=0$, namely $\gamma_1=0$. 
    It follows that $d\alpha^n=c{\al11}$, so in particular $d\eta\in I\pt{\al11}$ and
    $$
    d\Omega_{\mathfrak h}=\alpha^1\wedge d\eta+\alb1\wedge d\overline{\eta}=0.
    $$
    From the first line of (\ref{dOmega22}) we get $\omega \wedge d\alpha^n=0$.
    Using again the positive definiteness of $\omega $, we obtain that $d\alpha^n=0$ and $\mathfrak g$ is abelian.
\end{proof}

\section{$(n-2)$-K\"ahler almost abelian solvmanifolds}\label{secAlmostAb}

{In this section we will  discuss the case where $(X, J)$ is   an almost abelian solvmanifold, i.e. a compact  quotient $\Gamma\backslash G$ of a simply connected almost abelian  Lie group $G$ by a lattice $\Gamma$ endowed with an invariant complex structure $J$.
Lemma \ref{symmetrization} allows us to restrict to the study of $p$-K\"ahler structures on unimodular almost abelian Lie  algebras.}

{Let $\mathfrak g$ be an almost abelian Lie algebra of real  dimension ${2n}$ and denote with $\mathfrak a$ its codimension one abelian ideal. 
Let  $(J, g)$  be a Hermitian  structure on $\frak g$ and denote by $\mathfrak a_1$ the $J$-invariant space $\mathfrak a\cap J\mathfrak a$.
Then there exists a unitary basis  $\pg{e_1, \dots , e_{2n}}$  
such that $\mathfrak a=\text{span}\pg{e_1, \dots,  e_{2n-1}}$, $\mathfrak a_1 = \text{span}\pg{e_2, \dots,  e_{2n-1}} $  and $Je_j=e_{2n+1-j}$,  for $j=1, \dots,  n$. 
The matrix associated to $\rest{\text{ad}_{e_{{2n}}}}{\mathfrak   a} $ in this basis can be written as 
$$
\begin{pmatrix}
\lambda & 0\\
v & A
\end{pmatrix},
$$
with $\lambda\in\R$, $v\in\mathfrak a_1$, $A=(a_{j,k
})_{j,k=2}^{2n-1}\in\mathfrak{gl}\pt{\mathfrak a_1}$.
{We will refer to such a basis $\pg{e_1\dots e_{2n}}$ as \textit{adapted} to the Hermitian structure $(J,g)$.}
By \cite{LR,AL}  the integrability  of $J$  is equivalent to $AJ_1=J_1A$, with $J_1\coloneqq\rest J{\mathfrak a_1} $, so that $A$ must satisfy
\begin{equation*}
a_{{2n+1}-j,k}=  \begin{cases}
                -a_{j,{2n+1}-k}  &   k=2,\dots,{n}   \\
                a_{j,{2n+1}-k}   &   k={n+1},\dots,{2n-1}
            \end{cases}
\quad\quad j=2,\dots,{n}.
\end{equation*}
If $\pg{e^1\dots e^{2n}}$ is the dual basis to $\pg{e_1\dots e_{2n}}$, we have that  $\alpha^j=e^j+ie^{2n+1-j}$, for $j=1\dots n$, is a basis of $(1,0)$-forms on $\frak g$  and the complex structure equations of $(\mathfrak g, J)$ are
\begin{equation}\label{genalmabcxstr}
\begin{cases}
    \displaystyle d\alpha^1=\frac{i}{2}\lambda\al11,   &   \\
    \displaystyle d\alpha^j=\frac{i}{2}w_j\al11+\frac{\alpha^1-\alpha^{\bar 1}}{2}\wedge\sum_{k=2}^nb_{jk}\alpha^k,    & j=2,\dots,n.
\end{cases}
\end{equation} 
with $w_j=v_j+i\,v_{{2n+1}-j}$ and $b_{jk}\coloneqq i\,a_{j,k}-a_{{2n+1}-j,k}=i\,a_{j,k}+a_{j,{2n+1}-k}$.

\begin{remark}\label{spMetAAb}
Notice that $\mathfrak g$ is unimodular if and only if $\lambda=-\text{tr}(A)$.
Moreover, $(J,g)$ is K\"ahler if and only if $v=0$ and $A=-A^t$ (see \cite{LR} and \cite[Lemma 3.6]{fp21}).
\end{remark}

We will now prove the following.

\begin{theorem}\label{ThmAlmostAb}
Let  $(\mathfrak g,J)$ be  a unimodular almost abelian Lie algebra of real  dimension $2n\geq6$ endowed with a complex structure $J$.
 If  $(\frak g,  J)$ admits a  $(n-2)$-K\"ahler structure, then   $(\mathfrak g,J)$ is K\"ahler.
\end{theorem}

\begin{proof}  We know that there exists   a basis  of $(1,0)$-forms $\alpha^j=e^j+ie^{2n+1-j}$, for $j=1\dots n$,   such that  the complex structure equations of $(\mathfrak g, J)$ are given by 
\eqref{genalmabcxstr}.
Suppose  that $(\mathfrak g,J)$ admits a $(n-2)$-K\"ahler form $\Omega$. 
Then, we can write
    \begin{equation}
    \Omega=\Theta+\alpha^1\wedge\eta+\alpha^{\bar1}\wedge\overline{\eta}+i\al11\wedge\vartheta,
    \end{equation}
with $\Theta\in\Lambda^
{{n-2},{n-2}}_{\mathfrak a_1}$, $\vartheta\in\Lambda^
{{n-3},{n-3}}_{\mathfrak a_1}$ both real and transverse and $\eta\in\Lambda^
{{n-3},{n-2}}_{\mathfrak a_1}$.
The restriction $\Theta$ of $\Omega$ to $\mathfrak a_1$ is a $(n-2,n-2)$-transverse form on a space of complex dimension $n-1$, so it is strictly positive and so  there exists a  new basis $\pg{\beta^2, \dots, \beta^n}$ of $\mathfrak a_1^{1,0}$ such that 
$$
\Theta= \left( i(\beta^{2 \overline 2}+\cdots+ \beta^{n \overline n}) \right)^{n-2}
$$
{\cite{michelsohn}}.
We can then  consider as basis $\{ f_j \}$ of $\mathfrak a_1$  the dual basis   of the basis    $\{ f^j \}$ of $\mathfrak a_1^*$, given by 
$$
f^j \coloneqq\frac{\beta^j+\overline \beta j}2,\quad f^{2n+1-j}\coloneqq -i\frac{\beta^j-\overline \beta^j}2,\quad j=2\dots n.
$$
Note that we  still have $Jf_j=f_{2n+1-j}$ for $j=2\dots n$.
We can complete $\pg{f_2, \dots,  f_{2n-1}}$ to a basis  $\pg{f_1, \dots,  f_{2n}}$ of $\mathfrak g$ just taking $f_1=e_1$ and $f_{2n}=e_{2n}$.
Then we have $\mathfrak a=\text{span}\pg{f_1, \dots,  f_{2n-1}}$ and $J f_1 = f_{2n}$.
Therefore the complex structure equations of $(\frak g, J)$ are still of the form 
\begin{equation*}
\begin{cases}
    \displaystyle d\beta^1= \frac{i}{2}\tilde \lambda\beta^{1 \overline 1},   &   \\
    \displaystyle d\beta^j= \frac{i}{2}\tilde w_j\beta^{1 \overline 1} +\frac{\beta^1-\beta^{\bar 1}}{2}\wedge\sum_{k=2}^n \tilde b_{jk}\beta^k,    & j=2,\dots,n,
\end{cases}
\end{equation*} 
with $\tilde \lambda = \lambda$ and $\beta^1 = \alpha^1$.
The
 pair $(\mathfrak g,J)$ is then determined by
the matrix associated {to $\rest{\text{ad}_{e_{{2n}}}}{\mathfrak   a} $} in this  new basis
$$
\begin{pmatrix}
\lambda     & 0     \\
\tilde v    & \tilde  A
\end{pmatrix},
$$
with $\tilde v\in\mathfrak a_1$, $\tilde A=( \tilde a_{j,k
})_{j,k=2}^{2n-1}\in\mathfrak{gl}\pt{\mathfrak a_1}$, and with $\tilde w_j= \tilde v_j+i\, \tilde v_{{2n+1}-j}$ and $ \tilde b_{jk}\coloneqq i\,\tilde a_{j,k}- \tilde a_{{2n+1}-j,k}=i\, \tilde a_{j,k}+ \tilde a_{j,{2n+1}-k}$.
Recall that the integrability  of $J$  is equivalent to  the condition $\tilde AJ_1=J_1 \tilde A$.
Therefore  $\tilde A$ must satisfy
\begin{equation*}
\tilde a_{{2n+1}-j,k}=  \begin{cases}
                -   \tilde a_{j,{2n+1}-k}  &   k=2,\dots,{n}   \\
                \tilde a_{j,{2n+1}-k}   &   k={n+1},\dots,{2n-1}
            \end{cases}
\quad\quad j=2,\dots,{n}.
\end{equation*}
Moreover,  
\begin{equation*}\begin{split}
    d\Theta=i\pt{\beta^1-\beta^{\bar1}}\wedge\rho+i \beta^{1 \overline 1}\wedge\gamma, \\
    d\eta=i\pt{\beta^1-\beta^{\bar1}}\wedge\phi+\beta^{1 \overline 1}\wedge\psi
\end{split}  \end{equation*}
with $\rho \in\Lambda^{{n-2},{n-2}}_{\mathfrak a_1}$ and $\gamma\in\Lambda^{2n-5}_{\mathfrak a_1}$ real forms.
Notice that $\rho$ and $\phi$ depend on $\tilde A$, while $\gamma$ and $\psi$ depend on {$\tilde v$}.
We have 
\begin{equation*}\begin{split}
   0= d\Omega&=d\Theta+{\lambda} \frac{i}{2}\beta^{1 \overline 1}\wedge\pt{\eta+\bar\eta}-\beta^1\wedge    d\eta-\beta^{\bar1}\wedge\overline{d\eta},   \\
    &=i\beta^{1\overline 1} \wedge\pt{\gamma+\frac{{\lambda}}{2}\pt{\eta+\bar\eta}+\phi+\bar\phi}+\pt{\beta^1-\beta^{\bar1}}\wedge\rho,
\end{split}\end{equation*}
or equivalently
\begin{equation}\label{Omegaclosed}
\rho=0,\quad\quad\gamma+\frac{\lambda}{2}\pt{\eta+\bar\eta}+\phi+\bar\phi=0.
\end{equation}
As a first consequence, we have that the theorem is true when $\tilde A=0$.
In this case, $\rho=0$, $\phi=0$ and $\mathfrak g$ is unimodular if and only if $\lambda=0$, so that \eqref{Omegaclosed} reduces to $\gamma=0$, namely $d\Theta=0$.
Because $d\beta^{1\bar1}=0$ and $d\varphi\in{I}(\beta^{1\bar1})$ for all $\varphi\in\Lambda_{\mathfrak g}$, one has 
$$
0=d\Theta=d\left( i(\beta^{1 \overline 1}+\cdots+ \beta^{n \overline n}) \right)^{n-2},
$$
hence $d(\beta^{1 \overline 1}+\cdots+ \beta^{n \overline n})=0$ (see \cite{GH}) and $(\mathfrak g,J)$ is K\"ahler.
We can now  prove that when $\tilde A\neq0$ the vanishing of $\rho$  implies that $\tilde A=- \tilde A^t$.
Indeed, we have $d\Theta=i^{n-2}(n-2)\pt{\beta^{2 \overline 2}+\cdots+\beta^{n \overline n}}^{n-3}\wedge\,d(\beta^{2 \overline 2} +\cdots+\beta^{n \overline n})$.
    For $j=2\dots n$, 
    $$
    d\beta^{j \overline j}=\frac{i}{2}\beta^{1 \overline 1}\wedge\pt{\tilde w_j\overline \beta j-\overline{\tilde w_j}\beta^j}+\frac{\beta^1-\beta^{\bar 1}}{2}\wedge\sum_{k=2}^n\pt{\tilde b_{jk}\beta^{k \overline j}-\overline{\tilde b_{jk}}\beta^{j \overline k}}
    $$
    For some $c_n\in\R$, we can write
    $$
    \pt{\beta^{2 \overline 2}+\cdots+\beta^{n \overline n}}^{n-3}=c_n\sum_{2\leq l<m\leq n}\beta^{2\bar 2\cdots\widehat{l\bar l}\cdots\widehat{m\bar m}\cdots{n\bar n} }
    $$
    so that 
    $$
    d\Theta=i\beta^{1 \overline 1}\wedge\gamma+\Tilde{c_n}\pt{\beta^1-\beta^{\bar1}}\wedge\sum_{2\leq l<m\leq n}\beta^{2\bar 2\cdots\widehat{l\bar l}\cdots\widehat{m\bar m}\cdots{n\bar n} }\wedge\sum_{j,k=2}^n\pt{\tilde b_{jk}\beta^{ k \overline j}-\overline{\tilde b_{jk}}\beta^{j \overline k}}
    $$
    with $\Tilde{c_n}\in\C$.
  We can rewrite the last sum as
    \begin{equation*}\begin{split}
  \sum_{j=2}^n&\beta^{ i\overline j}\pt{ \tilde b_{jj}-\overline{\tilde b_{jj}} }  +\sum_{2\leq j<k\leq n}\left( \pt{\tilde b_{jk}\al kj-\overline{\tilde b_{jk}}\beta^{ j \overline k}}+\pt{b_{kj}\beta^{ j \overline k}-\overline{b_{kj}}\beta^{k \overline j}}  \right)\\
  &=\sum_{j=2}^n\beta^{j \overline j} \pt{\tilde b_{jj}-\overline{\tilde b_{jj}} }  +\sum_{2\leq j<k\leq n}\left( \beta^{ j \overline k}\pt{ \tilde b_{kj}-\overline{\tilde b_{jk}}}+\beta^{k \overline j} \pt{\tilde b_{jk}-\overline{\tilde b_{kj}}}  \right)  \\
 &=\sum_{j=2}^n2i\, \tilde a_{jj}\beta^{j  \overline j}  +\sum_{2\leq j<k\leq n}\left( \beta^{j \overline k}\pt{\tilde b_{kj}-\overline{\tilde b_{jk}}}+\beta^{k \overline j}\pt{\tilde b_{jk}-\overline{\tilde b_{kj}}}  \right)  
   \end{split}            \end{equation*}
 It follows that up to a complex constant, $\rho$ equals
    \begin{equation*}
    \sum_{l\neq m=2}^n{2i\tilde a_{ll}}\beta^{2\bar 2\cdots\widehat{m\bar m}\cdots{n\bar n} }    +
    \sum_{2\leq l<m\leq n}\beta^{2\bar 2\cdots\widehat{l\bar l}\cdots\widehat{m\bar m}\cdots{n\bar n} }\wedge\left( \beta^{l \overline m}\pt{\tilde b_{ml}-\overline{\tilde b_{lm}}}+\beta^{m \overline l} \pt{\tilde b_{lm}-\overline{\tilde b_{ml}}}  \right) 
    \end{equation*}
    If $\rho=0$, one gets
    \begin{equation*}\begin{cases}
    \displaystyle\sum_{l\neq m} \tilde a_{ll}=0,&   m=2\dots n,     \\
    \tilde b_{lm}=\overline{\tilde b_{ml}},            &   2\leq l<m\leq n ,
    \end{cases}
    \end{equation*} 
 namely
    \begin{equation*}
    \tilde a_{ll}=0,\quad \tilde a_{lm}=- \tilde a_{ml},\quad \tilde a_{{2n+1}-l,m}= \tilde a_{{2n+1}-m,l}=- \tilde a_{m,{2n+1}-l}.
    \end{equation*}
    This, together with the conditions on $\tilde A$ given by the integrability of the complex structure, is enough to conclude that $\tilde A=- \tilde A^t$.  
    Therefore, the matrix associated {to $\rest{\text{ad}_{e_{{2n}}}}{\mathfrak   a} $} in the basis $\pg{f_1, \dots,  f_{2n-1}}$ is 
    $$
    C =   \begin{pmatrix}
      {\lambda} &   0   \\
       \tilde  v       &   \tilde A   \\
    \end{pmatrix},
    $$
    with $\tilde A$ antisymmetric.
    Because $\tilde A\neq0$, the Jordan form of $C$ is 
    $$
    \text{Jord}(C)=\begin{pmatrix}
        \lambda &   \delta  \,\, 0 \cdots 0  \\
        0       &   \text{Jord}(\tilde A)   \\
    \end{pmatrix},
    $$
    where $\text{Jord}({\tilde A})$ is the Jordan form of ${\tilde A}$, $\delta=0$ if $\lambda$ is an eigenvalue of $\tilde A$ and $\delta=1$ otherwise.
    It follows that $C$ is similar to a matrix 
    $$
    D=\begin{pmatrix}
    \lambda &   0   \\
    0       &   \tilde D   \\
    \end{pmatrix},
    $$
with $\tilde D$ antisymmetric and $\text{Jord}({\tilde D})=\text{Jord}({\tilde A})$. 
Consider the almost abelian Lie algebra $\tilde{\mathfrak g}$ with abelian ideal $\tilde{\mathfrak a}=\text{span}\pg{\tilde e_1\dots\tilde{e}_{2n-1}}$ and such that the matrix of $\rest{\text{ad}_{\tilde e_{2n}}}{\tilde{\mathfrak a}}$ is $D$. 
By \cite[Proposition 1]{Freibert}, $\tilde{\mathfrak g}$ is isomorphic to $\mathfrak g$.
As mentioned above (Remark \ref{spMetAAb}), this gives {the existence of a  K\"ahler metric} on $\tilde{\mathfrak g}$.
\end{proof}}

\begin{remark}
    In complex dimension $4$, the theorem states that there are no $2$-K\"ahler almost abelian solvmanifolds that are non--K\"ahler.
\end{remark}

\begin{remark}
In the last part of the proof we found a sufficient condition for an almost abelian unimodular Lie algebra $\mathfrak g$ to be K\"ahler.
Let $(J,g)$ be a Hermitian structure on $\mathfrak g$ and $\pg{e_1\dots e_{2n}}$ be an adapted basis to $(J,g)$.
If $\rest{\text{ad}_{e_{2n}}}{\mathfrak a}$ is conjugated to a matrix of the form
$$
\begin{pmatrix}
\lambda & 0\\
v & A
\end{pmatrix},
$$
with $\lambda\in\R$, $v\in\mathfrak a_1$, $A\in\mathfrak{so}(\mathfrak a_1)$, $\pq{A,\rest J {\mathfrak a_1}}=0$, and $A$ has same rank of $(v\, A)$, then $(\mathfrak g,J)$ is K\"ahler.
\end{remark}

{\bf Acknowledgements.}
Anna Fino and Asia Mainenti are partially supported by Project PRIN 2017 “Real and complex manifolds: Topology, Geometry and Holomorphic Dynamics” and by GNSAGA (Indam). Anna Fino is also supported   by a grant from the Simons Foundation (\#944448). Asia Mainenti would like to thank the Department of Mathematics and Statistics of Florida International University for the hospitality.
The authors would like to thank Adrián Andrada, Elia Fusi, Adela Latorre, Nicoletta Tardini  and the anonymous referee for useful comments.

\end{document}